\documentclass[11pt]{amsart}
\usepackage{amsmath}
\usepackage{amssymb}
\newtheorem{theorem}{Theorem}[section]
\theoremstyle{plain}

\newtheorem{definition}{Definition}[section]
\newtheorem{example}{Example}[section]

\newtheorem{lemma}{Lemma}[section]

\newtheorem{proposition}{Proposition}[section]
\newtheorem{remark}{Remark}[section]

\numberwithin{equation}{section}

\begin{document}
\title[Closedness and D-connectedness in QV-Approach Spaces]{The Notion of Closedness and D-connectedness in Quantale-valued Approach Spaces}

\author{Muhammad Qasim}
\curraddr{Department of Mathematics\\ National University of Sciences \& Technology (NUST)\\ H-12 Islamabad, 44000 \\ Pakistan}
\email{muhammad.qasim@sns.nust.edu.pk \\ qasim99956@gmail.com}

\author{Samed \"{O}zkan}
\curraddr{Department of Mathematics\\ Nev\c{s}ehir Hac{\i} Bekta\c{s} Veli University \\
	50300 Nev\c{s}ehir \\ Turkey }
\email{ozkans@nevsehir.edu.tr}

\date{}



\begin{abstract}
In this paper, we characterize the local $T_{0}$ and $T_{1}$ separation axioms for quantale-valued gauge space, show how these concepts are related to each other and apply them to $L$-approach space and $L$-approach system.  Furthermore, we give the characterization of a closed point and $D$-connectedness in quantale-valued gauge space. Finally, we compare all these concepts with other.
\end{abstract}

\maketitle

\section{Introduction}
Approach spaces have been introduced by Lowen \cite{Lowen1989, Lowen1997} to generalize metric and topological concepts, have many applications in almost all areas of mathematics including probability theory \cite{approachProbability}, convergence theory \cite{ConvApp}, domain theory \cite{AppDomain} and fixed point theory \cite{fixedpointApp}. Due to its huge importance, several generalization of approach spaces appeared recently such as probabilistic approach spaces \cite{probApp}, quantale-valued gauge spaces \cite{QuantaleGauge}, quantale-valued approach system \cite{QuantaleSystem} and quantale-valued approach w.r.t closure operators \cite{TholenApp}. Recently, some quantale-valued approach spaces \cite{L-AppJager} are also characterized by using quantale-valued bounded interior spaces and bounded strong topological spaces which are commonly used by fuzzy mathematicians.

In 1991, Baran \cite{BaranSep} introduced local separation axioms and notion of closedness in set-based topological categories which are used to define several distinct Hausdorff objects \cite{BaranT2}, $T_{3}$  and $T_{4}$ objects \cite{BaranT3}, regular, completely regular, normed objects \cite{BaranNormal}, notion of compactness and minimality, perfectness \cite{BaranPerfect} and closure operators in the sense of Guili and Dikranjan \cite{Closure} in some well-known topological categories \cite{BaranPreordClos, BaranSemiunifClos}.  

The main objective of this paper is:
\begin{enumerate}
	\item[(i)] to characterize local $T_{0}$ and local $T_{1}$ quantale-valued gauge spaces, quantale-valued distance-approach and quantale-valued approach systems, show their relationship to each other,
	\item[(ii)]  to provide the characterization of the notion of closedness and $D$-connectedness in quantale-valued approach space and how are linked to each other,
	\item[(iii)] to give a comparison between local $T_{0}$ and $T_{1}$ quantale-valued approach spaces, and between notion of closedness and $D$-connectedness, and to examine their relationships.
	\end{enumerate}

\section{Preliminaries}
Recall that for any non-empty set $L$ with $\leq$ order is partial order set (poset) if it satisfies for all $a, b, c \in L$, $a \leq a$ (reflexivity), $a \leq b \wedge b \leq a \Rightarrow a =b$ (anti-symmetry) and $a \leq b \wedge b \leq c \Rightarrow a \leq c$ (transitivity). For any poset, the top or (maximum) element and bottom or (minimum) element is denoted by $\top$ and $\bot$ respectively. $(L, \leq)$ is complete lattice if all subsets of $L$ have both supremum ($\bigvee$) and infimum ($\bigwedge$), and \textit{well-below relation} $\alpha \lhd \beta$ if all subsets $A \subseteq L$ such that $\beta \leq \bigvee A$ there is $\delta \in A$ such that $\alpha \leq \delta$. Similarly, \textit{well-above relation} $\alpha \prec \beta $ if all subsets $A \subseteq L$ such that $\bigwedge A \leq \alpha$ there exists $\delta \in A$ such that $\delta \leq \beta$.  Furthermore, $L$ is completely distributive lattice iff for any $\alpha \in L$, $\alpha = \bigvee\{\beta : \beta \lhd \alpha\}$.

The triple $(L, \leq, *)$ is called \textit{quantale} if $(L, *)$ is a semi group, and operation $*$ satisfies the following properties: for all $\alpha_{i}, \beta \in L$: $(\bigvee_{i \in I} \alpha_{i})* \beta = \bigvee_{i \in I} (\alpha_{i} * \beta)$ and $\beta * (\bigvee_{i \in I} \alpha_{i}) = \bigvee_{i \in I} (\beta * \alpha_{i})$ and $(L, \leq)$ is a complete lattice.

A quantale $(L, \leq, *)$ is \textit{commutative} if $(L, *)$ is a commutative semi group and \textit{integral} if for all $\alpha \in L$ : $\alpha * \top = \top * \alpha = \alpha$.

\begin{definition} (cf. \cite{QuantaleGauge})
An $\textit{L-metric}$ on a set $X$ is a map $d: X \times X \rightarrow (L, \leq, *)$ satisfies for all $x \in X$, $d(x,x)= \top$, and for all $x,y,z \in X$, $d(x,y) * d(y,z) \leq d(x,z)$.	
\end{definition}
A map $f:(X, d_{X}) \rightarrow (Y, d_{Y})$ is a $L-$metric morphism if for all $x_{1}, x_{2} \in X$, $d_{X}(x_{1}, x_{2}) \leq d_{Y}(f(x_{1}), f(x_{2}))$.

$\textbf{L-MET}$ is category whose objects are $L-$metric spaces and morphisms are $L-$metric morphism.
\begin{example}
	\begin{enumerate}
		\item[(i)] If $L =\{0,1\}$, then $L-$metric space is a preordered set.
		\item[(ii)] If $(L= [0, \infty], \geq, +)$, then $L-$metric space is extended pseudo-quasi metric space. 
	\end{enumerate}
\end{example}

\begin{definition} (cf. \cite{QuantaleGauge})
	Let $\mathcal{H} \subseteq \textbf{L-MET(X)}$ and $d \in \textbf{L-MET(X)}$.
	\begin{enumerate}
		\item[(i)] $d$ is called \textit{locally supported} by $\mathcal{H}$ if for all $x \in X$, $\alpha \lhd \top$, $\bot \prec \omega$, there is $e \in \mathcal{H}$ such that $e(x,.)* \alpha \leq d(x, .) \vee \omega$.
		\item[(ii)] $\mathcal{H}$ is called \textit{locally directed} if for all subsets $\mathcal{H}_{0} \subseteq \mathcal{H}$, $\bigwedge_{d \in \mathcal{H}_{0}} d$ is locally supported by $\mathcal{H}$.
		\item[(iii)] $\mathcal{H}$ is called \textit{locally saturated} if for all $ d \in \textbf{L-MET(X)}$, we have $d \in \mathcal{H}$ whenever $d$ is locally supported by $\mathcal{H}$.
		\item[(iv)] The set $\tilde{\mathcal{H}} = \{d \in \textbf{L-MET(X)}: d$ is locally supported by $ \mathcal{H} \}$ is called \textit{local saturation} of $\mathcal{H}$.
	\end{enumerate}
\end{definition}

\begin{definition} (cf. \cite{QuantaleGauge})
	Let $X$ be a set. $\mathcal{G} \subseteq \textbf{L-MET(X)}$ is called an $L$-gauge if $\mathcal{G}$ satisfies the followings:
	\begin{enumerate}
		\item[(i)] $\mathcal{G} \not = \emptyset$.
		\item[(ii)] $d \in \mathcal{G}$ and $d \leq e$ implies $e \in \mathcal{G}$.
		\item[(iii)] $d, e \in \mathcal{G}$ implies $d \wedge e \in \mathcal{G}$.
		\item[(iv)] $\mathcal{G}$ is locally saturated.
	\end{enumerate}
The pair $(X, \mathcal{G})$ is called an $L$-gauge space.
\end{definition}
A map $f:(X, \mathcal{G}) \rightarrow (X', \mathcal{G}')$ is a $L$-gauge morphism if $d' \circ (f \times f) \in \mathcal{G}$ when $d' \in \mathcal{G}'$. 

The category whose objects are $L$-gauge spaces and morphisms are $L$-gauge morphisms is denoted by $\textbf{L-GS}$ (cf. \cite{QuantaleGauge}).  

\begin{proposition} \label{locally directed L-gauge} (cf. \cite{QuantaleGauge})
	Let $(L, \leq, *)$ be a quantale. If $\emptyset \not = \mathcal{H} \subseteq \textbf{L-MET(X)}$ is locally directed, then $\mathcal{G} = \tilde{\mathcal{H}}$ is a gauge with $\mathcal{H}$ as a basis.
\end{proposition}
\begin{definition} (cf. \cite{QuantaleGauge})
	A map $\delta: X \times P(X) \rightarrow (L, \leq, *)$ is called $L$-approach if $\delta$ satisfies the followings:
	\begin{enumerate}
		\item[(i)] $\forall x \in X$: $\delta(x, \{x\}) = \top$.
		\item[(ii)] $\forall x \in X$ and $\emptyset$, the empty set: $\delta(x, \emptyset) = \bot$.
		\item[(iii)] $\forall x \in X$ and $\forall A, B \subseteq X$: $\delta(x, A \cup B) = \delta(x, A) \vee \delta(x, B)$.
		\item[(iv)]  $\forall x \in X$, $\forall A\subseteq X$ and $\forall \alpha \in L$: $\delta(x, A) \geq \delta(x, \overline{A}^{\alpha}) * \alpha$, where $\overline{A}^{\alpha} = \{x \in X : \delta(x, A) \geq \alpha \}$.
	\end{enumerate}
The pair $(X, \delta)$ is called $L$-Approach distance space. 
\end{definition}
A map $f:(X, \delta) \rightarrow (X', \delta')$ is a $L$-approach morphism if for all $x \in X$ and $A \subseteq X$: $\delta(x, A) \leq \delta'(f(x), f(A))$. 

The category whose objects are $L$-approach distance space and morphisms are $L$-approach morphisms is denoted by $\textbf{L-AP}$.

\begin{definition} (cf. \cite{QuantaleSystem})
	Let $\mathcal{A} \subseteq L^{X}$ and $\varphi \in L^{X}$.
	\begin{enumerate}
		\item[(i)] $\varphi$ is \textit{supported} by $\mathcal{A}$ if for all $\alpha \lhd \top$, $\bot \prec \omega$ there exists $\varphi_{\alpha}^{\omega} \in \mathcal{A}$ such that $\varphi_{\alpha}^{\omega}*\alpha \leq \varphi \vee \omega$.
		\item[(ii)] $\mathcal{A}$ is \textit{saturated} if $\varphi \in \mathcal{A}$ whenever $\varphi$ is supported by $\mathcal{A}$.
		\item[(iii)] For $\mathcal{B} \subseteq L^{X}$, $\tilde{\mathcal{B}} = \{\varphi \in L^{X}: \varphi$ is supported by $\mathcal{B}\}$ is called saturation of $\mathcal{B}$.
	\end{enumerate}
\end{definition}

\begin{definition} (cf. \cite{QuantaleSystem})
	Let $\mathcal{A}(x) \subseteq L^{X}$, for all $x \in X$. Then, $\mathcal{A} = (\mathcal{A}(x))_{x \in X}$ is called an $L$-\textit{approach system} if for all $x \in X$,
	\begin{enumerate}
		\item[(i)] $\mathcal{A}(x)$ is a filter in $L^{X}$, i.e., $\varphi \in \mathcal{A}(x)$ and $\varphi \leq \varphi'$ implies $\varphi' \in \mathcal{A}(x)$, and $\varphi, \varphi' \in \mathcal{A}(x)$ implies $\varphi \wedge \varphi' \in \mathcal{A}(x)$.
		\item[(ii)] $\varphi (x) = \top$ whenever $\varphi \in \mathcal{A}(x)$.
		\item[(iii)] $\mathcal{A}(x)$ is saturated.
		\item[(iv)] For all $\varphi \in \mathcal{A}(x)$, $\alpha \lhd \top$, $\bot \prec \omega$ there exists a family $(\varphi_{z})_{z \in X} \in \prod_{z \in X} \mathcal{A}(z)$ such that $\varphi_{x}(z) * \varphi_{z}(y) * \alpha \leq \varphi(y) \vee \omega$, $\forall y,z \in X$. 
	\end{enumerate}
The pair $(X, \mathcal{A})$ is called an $L$-\textit{approach system space}.
\end{definition} (cf. \cite{QuantaleSystem})
The map $f : (X, \mathcal{A}) \rightarrow (X', \mathcal{A}')$ is a $L$-\textit{approach system morphism} if for all $x \in X$, $\varphi' \circ f \in \mathcal{A}(x)$ when $\varphi' \in \mathcal{A}'(f(x))$.

The category whose objects are $L$-approach system spaces and morphisms are $L$-approach system morphisms is denoted by $\textbf{L-AS}$.

\begin{definition} (cf. \cite{QuantaleSystem})
	Let $\mathcal{B}(x) \subset L^{X}$, for all $x \in X$. Then,  $(\mathcal{B}(x))_{x \in X}$ is called $L$-\textit{approach system basis} if for all $x \in X$,
	\begin{enumerate}
		\item[(i)] $\mathcal{B}(x)$ is a filter basis in $L^{X}$.
		\item[(ii)] $\varphi (x) = \top$ whenever $\varphi \in \mathcal{B}(x)$.
		\item[(iii)] For all $\varphi \in \mathcal{B}(x)$, $\alpha \lhd \top$, $\bot \prec \omega$ there exists a family $(\varphi_{z})_{z \in X} \in \prod_{z \in X} \mathcal{B}(z)$ such that $\varphi_{x}(z) * \varphi_{z}(y) * \alpha \leq \varphi(y) \vee \omega$, $\forall y,z \in X$.
	\end{enumerate}
\end{definition}

\begin{definition} (cf. \cite{QuantaleSystem}) \label{L-App Base}
	Let  $(\mathcal{A}(x))_{x \in X}$ is called an $L$-approach system and $(\mathcal{B}(x))_{x \in X}$ be the collection of filter bases on $L^{X}$. Then, $(\mathcal{B}(x))_{x \in X}$ is called \textit{base} for $L$-approach system if $\tilde{\mathcal{B}}(x) = \{\varphi \in L^{X}: \varphi$ is supported by $\mathcal{B}(x)\}= \mathcal{A}(x)$.
\end{definition}

Note that the categories $\textbf{L-GS}$, $\textbf{L-AP}$ and $\textbf{L-AS}$ are equivalent categories \cite{QuantaleGauge, QuantaleSystem}, and they are all topological categories on $\textbf{Set}$ and we will denote any $L$-approach space by $(X, \mathfrak{B})$.

\begin{remark} (cf. \cite{QuantaleGauge, QuantaleSystem}) \label{Transition}
	\begin{enumerate}
		\item[(i)] The transition from $L$-gauge to $L$-approach is determined by
		\begin{equation*}
		\delta(x,A)= \bigwedge_{d \in \mathcal{G}} \bigvee_{a \in A} d(x,a)
		\end{equation*}
		\item[(ii)] The transition from $L$-approach distance to $L$-approach system is defined by
		\begin{equation*}
	\mathcal{A}(x)= \{\varphi \in L^{X}: \delta(x,A) \leq \bigvee_{a \in A} \varphi(a),\forall A \subseteq X, x \in X\}
		\end{equation*}
		\item[(iii)] The Transition from $L$-approach system to $L$-gauge is given by
		\begin{equation*}
		\mathcal{G}^{\mathcal{A}} = \{d \in\textbf{L-MET(X)}: \forall x \in X, d(x,.) \in \mathcal{A}(x) \}
		\end{equation*}
		\item [(iv)] The transition from $L$-approach distance to $L$-gauge is given by 
		\begin{equation*}
		\mathcal{G} = \{d \in \textbf{L-MET(X)}: \forall A \subseteq X, x \in X : \delta(x,A) \leq \bigvee_{a \in A} d(x,a) \}
		\end{equation*}
		\item[(v)] The transition from $L$-approach system to $L$-approach distance is determined by 
		\begin{equation*}
		\delta(x,A)= \bigwedge_{\varphi \in \mathcal{A}(x)} \bigvee_{a \in A} \varphi(a)
		\end{equation*}
	\end{enumerate}
\end{remark}

\begin{lemma} (cf. \cite{QuantaleGauge, QuantaleSystem}) \label{initialStr}
	Let $(X_{i}, \mathfrak{B}_{i})$ be the collection of $L$-approach spaces and $f_{i}: X \rightarrow (X_{i}, \mathfrak{B}_{i})$ be a source and $x \in X$.
	\begin{enumerate}
		\item[(i)] The initial $L$-gauge base on $X$ is given by 
		\begin{equation*}
		\mathcal{H} = \{\bigwedge_{i \in K} d_{i} \circ (f_{i} \times f_{i}) : K \subseteq I finite, d_{i} \in \mathcal{G}_{i}, \forall i \in I \}
		\end{equation*}
		\item[(ii)] The initial $L$-approach system base is provided by
		\begin{equation*}
		\mathcal{B}(x) = \{\bigwedge_{i \in K} \varphi_{i} \circ f_{i} : K \subseteq I finite, \varphi_{i} \in \mathcal{A}_{i}(f_{i}(x)), \forall i \in I \}
		\end{equation*}
		
		\end{enumerate}
\end{lemma}  

\begin{proposition} \label{gdis}
Let $(X, \mathcal{G})$ be an $L$-gauge space. The discrete $L$-gauge structure on $X$ is $\mathcal{G}_{dis} = \textbf{L-MET(X)}$ (all $L$-metric spaces on $X$) and the base that induces $\textbf{L-MET(X)}$ consists of following discrete $L$-metric space:
\begin{equation*}
\begin{array}{rlll}
d_{dis}(x,y)=%
\begin{cases}
\top, & \text{$x=y$} \\
\bot , & \text{$x \not= y$}%
\end{cases}
&  &  &
\end{array}%
\end{equation*} 	
\end{proposition}

\begin{proof}
It can be easily seen that $d_{dis}$ is the smallest $L$-metric space on $X$. To show that $\mathcal{H} = \{d_{dis}\}$ is a $L$-gauge base, by proposition \ref{locally directed L-gauge}, it is sufficient to show that $\mathcal{H}$ is locally directed. Since a base consist of one element is always locally directed, $\mathcal{H} = \{d_{dis}\}$ is a $L$-gauge base and the generated $L$-gauge (by adding all locally supported $L$-metrics) is just the principle filters of $d$, i.e., $\mathcal{G} = \{e\in \textbf{L-MET(X)}: e\geq d_{dis}\}$. Therefore, all $L$-metrics are in $L$-gauge. Thus, $\mathcal{G} = \textbf{L-MET(X)}$. Furthermore, for any given $f:(X, \mathcal{G} = \textbf{L-MET(X)}) \rightarrow (X', \mathcal{G}')$ is a $L$-gauge morphism. Therefore, $\mathcal{G}_{dis}=\textbf{L-MET(X)}$.	
\end{proof}

\section{Local $T_{0}$ and $T_{1}$ Quantale-valued Approach Spaces}

	Let $X$ be a set and $p$ be a point in $X$. Let $X \amalg X$ be the coproduct of $X$. The \textit{wedge product} (``one point union'') of $X$ at $p$ is a pushout, and denoted by $X \vee_{p} X$. 

A point $x$ in $X \vee_{p} X$ is denoted as $x_{1}$ if it lies in first component and as $x_{2}$ if it lies in second component. 

Let $X^{2}$ be the cartesian product of $X$.

\begin{definition} (cf. \cite{BaranSep})
	A map $A_{p}: X \vee_{p} X \rightarrow X^{2}$ is called \textit{principal $p$-axis map} if
	\begin{equation*}
	\begin{array}{rlll}
	A_{p}(x_{i})=%
	\begin{cases}
	(x,p), & \text{ $i=1$} \\
	(p,x) , & \text{$i=2$}%
	\end{cases}
	&  &  &
	\end{array}%
	\end{equation*}
\end{definition}

\begin{definition}
	A map $S_{p}: X \vee_{p} X \rightarrow X^{2}$ is called \textit{skewed $p$-axis map} if
	\begin{equation*}
	\begin{array}{rlll}
	S_{p}(x_{i})=%
	\begin{cases}
	(x,x), & \text{ $i=1$} \\
	(p,x) , & \text{$i=2$}%
	\end{cases}
	&  &  &
	\end{array}%
	\end{equation*}
\end{definition}
\begin{definition}
	A map $\nabla_{p}: X \vee_{p} X \rightarrow X$ is called \textit{folding map at $p$} if  $\nabla_{p}(x_{i}) = x$ for $i=1,2$.
\end{definition}

Recall \cite{Category, Preuss, PreussCat}, that a functor $\mathcal{U}: \mathcal{E} \rightarrow \mathbf{Set}$ is called topological if $\mathcal{U}$ is concrete, consists of small fibers and each $\mathcal{U}$-source has an initial lift.

Note that a topological functor has a left adjoint called the discrete functor.

\begin{theorem} (cf. \cite{BaranSep}) \label{T0veT1Def}
	Let $U:\mathcal{E}\rightarrow \mathbf{Set}$ be
	topological, $X \in Ob(\mathcal{E})$ with $U(X)=Y$ and $p \in Y$.
	
	\begin{enumerate}
		\item [(i)] $X$ is $\overline{T_{0}}$ at $p$ iff the initial lift of the $U$%
		-source $\{A_{p}: Y \vee_{p} Y \rightarrow U(X^{2})=Y^{2}$ and $\nabla_{p}:Y\vee_{p} Y \rightarrow UD(Y)=Y \}$ is discrete, where $D$ is the discrete functor.
		\item[(ii)] $X$ is $T_{1}$ at $p$ iff the initial lift of the $U$%
		-source $\{S_{p}: Y \vee_{p} Y \rightarrow U(X^{2})=Y^{2}$ and $\nabla_{p}:Y\vee_{p} Y \rightarrow UD(Y)=Y \}$ is discrete, where $D$ is the discrete functor.
	\end{enumerate}	
\end{theorem}

\begin{remark}
	In $\mathbf{Top}$ (category of topological spaces and continuous map), $\overline{T_{0}}$ (resp. $T_{1}$) reduces to usual $\mathbf{T_{0}}$ (resp. $\mathbf{T_{1}}$), and $(X, \tau)$ is $\mathbf{T_{0}}$ (resp. $\mathbf{T_{1}}$) at $p$ iff $(X, \tau)$ is $\mathbf{T_{0}}$ (resp. $\mathbf{T_{1}}$) for all $p \in X$ \cite{BaranSepBalk}.
\end{remark}

\begin{theorem} \label{local T0 L-App}
	Let $(X,\mathcal{G})$ be an L-gauge space and $p \in X$. $(X,\mathcal{G})$ is $\overline{T_{0}}$ at $p$ if and only if for all $x\in X$ with $x\neq p$, there exists $d\in \mathcal{G}$ such that $d(x,p)=\bot $ or $d(p,x)=\bot $.
\end{theorem}

\begin{proof}
Suppose $(X,\mathcal{G})$ be $\overline{T_{0}}$ at $p$, $x\in X$ and $x\neq p$. Let $\overline{\mathcal{G}}$ be the initial L-gauge on $X\vee _{p}X$ induced by $A_{p}:X\vee _{p}X\rightarrow U(X^{2},\mathcal{G}^{2})=X^{2}$ and $\nabla _{p}:X\vee _{p}X\rightarrow U(X,\mathcal{G}_{dis})=X$ where $\mathcal{G}_{dis}$ is discrete structure on $X$ and $\mathcal{G}^{2}$ is product structure on $X^{2}$ induced by $\pi_{i}:X^{2} \rightarrow X$ projection maps for $i=1,2$. Suppose that $\mathcal{H}_{dis}=\{d_{dis}\}$ is a basis for discrete L-gauge where $d_{dis}$ is the discrete L-metric on $X$. Let $\mathcal{H}$ be an L-gauge basis of $\mathcal{G}$ and $d\in \mathcal{H}$, and $\overline{\mathcal{H}}=\{\overline{d}_{dis}\}$ be the initial L-gauge basis of $\overline{\mathcal{G}}$ where $\overline{d}_{dis}$ is the discrete L-metric on $X\vee _{p}X$. For $x_{1}, x_{2} \in X\vee _{p}X$ with $x_{1} \neq x_{2}$, note that
\begin{equation*}
d_{dis}(\nabla _{p}(x_{1}),\nabla _{p}(x_{2}))=d_{dis}(x,x)=\top
\end{equation*}
\begin{equation*}
d(\pi _{1}A_{p}(x_{1}),\pi _{1}A_{p}(x_{2}))=d(x,p)
\end{equation*}
\begin{equation*}
d(\pi _{2}A_{p}(x_{1}),\pi _{2}A_{p}(x_{2}))=d(p,x)
\end{equation*}
Since $x_{1}\neq x_{2}$, $\overline{d}_{dis}$ is the discrete L-metric on $X\vee _{p}X$ and $(X,\mathcal{G})$ is $\overline{T_{0}}$ at $p$, by Lemma \ref{initialStr}
\begin{eqnarray*}
	\bot &=&\overline{d}_{dis}(x_{1},x_{2}) \\
	&=&\bigwedge\ \{d_{dis}(\nabla _{p}(x_{1}),\nabla _{p}(x_{2})),d(\pi
	_{1}A_{p}(x_{1}),\pi _{1}A_{p}(x_{2})), \\
	& & \hspace{8mm} d(\pi _{2}A_{p}(x_{1}),\pi_{2}A_{p}(x_{2}))\} \\
	&=&\bigwedge\ \{\top,d(x,p),d(p,x)\}
\end{eqnarray*}
and consequently, $d(x,p)=\bot $ or $d(p,x)=\bot $.
\newline

Conversely, let $\overline{\mathcal{H}}$ be initial L-gauge basis on $X \vee_{p} X$ induced by $A_{p}: X \vee_{p} X \rightarrow U(X^{2}, \mathcal{G}^{2})=X^{2}$ and $\nabla_{p}: X \vee_{p} X \rightarrow U(X, \mathcal{G}_{dis})=X$ where, by Proposition \ref{gdis}, $\mathcal{G}_{dis}=\textbf{L-MET(X)} $ discrete L-gauge on $X$ and $\mathcal{G}^{2}$ be the product structure on $X^{2}$ induced by $\pi_{i}:X^{2} \rightarrow X$ the projection maps for $i=1,2$.

Suppose for all $x\in X$ with $x\neq p$, there exists $d\in \mathcal{G}$ such that $d(x,p)=\bot $ or $d(p,x)=\bot $. Let $\overline{d} \in \overline{\mathcal{H}}$ and $u, v \in X \vee_{p} X$.
\newline

If $u=v$, then
	\begin{eqnarray*}
		\overline{d}(u,v) &=&\bigwedge\ \{d_{dis}(\nabla _{p}(u),\nabla _{p}(u)),d(\pi
		_{1}A_{p}(u),\pi _{1}A_{p}(u)), \\
		& & \hspace{8mm} d(\pi _{2}A_{p}(u),\pi _{2}A_{p}(u))\} \\
		&=&\top
	\end{eqnarray*}

If $u\neq v$ and $\nabla _{p}(u)\neq \nabla _{p}(v)$, then $d_{dis}(\nabla _{p}(u),\nabla _{p}(v))=\bot $ since $d_{dis}$ is discrete. By Lemma \ref{initialStr}
	\begin{eqnarray*}
		\overline{d}(u,v) &=&\bigwedge\ \{d_{dis}(\nabla _{p}(u),\nabla _{p}(v)),d(\pi
		_{1}A_{p}(u),\pi _{1}A_{p}(v)), \\ 
		& & \hspace{8mm} d(\pi _{2}A_{p}(u),\pi _{2}A_{p}(v))\} \\
		&=&\bigwedge\ \{\bot, d(\pi _{1}A_{p}(u),\pi _{1}A_{p}(v)),d(\pi
		_{2}A_{p}(u),\pi _{2}A_{p}(v))\} \\
		&=&\bot
	\end{eqnarray*}

Suppose $u\neq v$ and $\nabla _{p}(u)=\nabla _{p}(v)$. If $\nabla _{p}(u)=x=\nabla _{p}(v)$ for some $x\in X$ with $x\neq p$, then $u=x_{1}$ and $v=x_{2}$ or $u=x_{2}$ and $v=x_{1}$ since $u\neq v$. Let $u=x_{1}$ and $v=x_{2}$.
	\begin{eqnarray*}
	d_{dis}(\nabla _{p}(u),\nabla _{p}(v)) &=&d_{dis}(\nabla _{p}(x_{1}),\nabla _{p}(x_{2})) \\
	&=& d_{dis}(x,x) \\
	&=&\top
	\end{eqnarray*}
	\begin{eqnarray*}
		d(\pi _{1}A_{p}(u),\pi _{1}A_{p}(v)) &=&d(\pi _{1}A_{p}(x_{1}),\pi
		_{1}A_{p}(x_{2})) \\
		&=&d(x,p)
	\end{eqnarray*}
	
	and
	\begin{eqnarray*}	
		d(\pi _{2}A_{p}(u),\pi _{2}A_{p}(v)) &=&d(\pi _{2}A_{p}(x_{1}),\pi _{2}A_{p}(x_{2})) \\
		&=&d(p,x)
	\end{eqnarray*}
	
	It follows that
	\begin{eqnarray*}
		\overline{d}(u,v) &=&\overline{d}(x_{1},x_{2}) \\
		&=&\bigwedge\ \{d_{dis}(\nabla _{p}(x_{1}),\nabla _{p}(x_{2})),d(\pi
		_{1}A_{p}(x_{1}),\pi _{1}A_{p}(x_{2})), \\
		& & \hspace{8mm} d(\pi _{2}A_{p}(x_{1}),\pi_{2}A_{p}(x_{2}))\} \\
		&=&\bigwedge\ \{\top,d(x,p),d(p,x)\}.
	\end{eqnarray*}
	By the assumption $d(x,p)$ $=\bot $ or $d(p,x)=\bot $, we get $\overline{d}(u,v)=\bot $.
	\newline
	
	Let $u=x_{2}$ and $v=x_{1}$. Similarly,
	\begin{eqnarray*}
		d_{dis}(\nabla _{p}(u),\nabla _{p}(v)) &=&d_{dis}(\nabla _{p}(x_{2}),\nabla _{p}(x_{1})) \\
		&=& d_{dis}(x,x) \\
		&=&\top
	\end{eqnarray*}
	\begin{eqnarray*}
		d(\pi _{1}A_{p}(u),\pi _{1}A_{p}(v)) &=&d(\pi _{1}A_{p}(x_{2}),\pi _{1}A_{p}(x_{1})) \\
		&=&d(p,x)
	\end{eqnarray*}
	
	and
	\begin{eqnarray*}	
		d(\pi _{2}A_{p}(u),\pi _{2}A_{p}(v)) &=&d(\pi _{2}A_{p}(x_{2}),\pi _{2}A_{p}(x_{1})) \\
		&=&d(x,p)
	\end{eqnarray*}
	
	It follows that
	\begin{eqnarray*}
		\overline{d}(u,v) &=&\overline{d}(x_{2},x_{1}) \\
		&=&\bigwedge\ \{d_{dis}(\nabla _{p}(x_{2}),\nabla _{p}(x_{1})),d(\pi
		_{1}A_{p}(x_{2}),\pi _{1}A_{p}(x_{1})), \\
		& & \hspace{8mm} d(\pi _{2}A_{p}(x_{2}),\pi_{2}A_{p}(x_{1}))\} \\
		&=&\bigwedge\ \{\top,d(p,x),d(x,p)\}.
	\end{eqnarray*}
	By the assumption $d(x,p)$ $=\bot $ or $d(p,x)=\bot $, we get $\overline{d}(u,v)=\bot $.
	\newline

Therefore, $\forall u,v\in X\vee _{p}X$, we have
\begin{equation*}
\begin{array}{r}
\overline{d}(u,v)=
\begin{cases}
\top, & \text{$u=v$} \\
\bot, & \text{$u\neq v$}
\end{cases}%
\end{array}%
\end{equation*}
and by the assumption $\overline{d}$ is the discrete L-metric on $X\vee _{p}X$, i.e., $\overline{\mathcal{H}}=\{\overline{d}\}$, which means $\mathcal{G}_{dis}=\textbf{L-MET(X)} $. By Definition \ref{T0veT1Def} (i), $(X,\mathcal{G})$ is $\overline{T_{0}}$ at $p$.			
\end{proof}

\begin{theorem} \label{AllQuantaleT0App}
Let $(X, \mathfrak{B})$ be a $L$-approach space and $p \in X$. Then, the followings are equivalent.
\begin{enumerate}
	\item[(i)] $(X, \mathfrak{B})$ is $\overline{T_{0}}$ at $p$.
	\item[(ii)] $\forall$ $x \in X$ with $x \not= p$, there exists $d\in \mathcal{G}$ such that $d(x,p)=\bot $ or $d(p,x)=\bot $.
	\item[(iii)] $\forall$ $x \in X$ with $x \not= p$, $\delta(x, \{p\})= \bot$ or $\delta(p, \{x\})= \bot$.
	\item[(iv)] $\forall$ $x \in X$ with $x \not= p$, there exists $\varphi \in \mathcal{A}(p)$ such that $\varphi (x)=\bot $ or  there exists $\varphi \in \mathcal{A}(x)$ such that $\varphi (p)=\bot $.
\end{enumerate}	
\end{theorem}

\begin{proof}
	 $(i) \Leftrightarrow (ii)$ follows from Theorem \ref{local T0 L-App}.
	 
	 $(ii) \Rightarrow (iii)$ Suppose $\forall x \in X$ with $x \not = p$, there exists $d \in \mathcal{G}$ such that $d(x, p)= \bot$ or $d(p, x)= \bot$. By Remark \ref{Transition} (i), $\delta(x, \{p\})= \bigwedge\limits_{d' \in \mathcal{G}} d'(x,p) = \bot$, and consequently, $\delta(x, \{p\}) = \bot$. Similarly, $\delta(p, \{x\})= \bigwedge\limits_{d' \in \mathcal{G}} d'(p,x) = \bot$ implies $\delta(p, \{x\}) = \bot$.
	 
	 $(iii) \Rightarrow (iv)$ Suppose $\forall x \in X$ with $x \not = p$, $\delta(x, \{p\}) = \bot$ or $\delta(p, \{x\}) = \bot$. Let $A = \{p\}$, then by Remark \ref{Transition} (ii), $\forall \varphi' \in L^{X}$, $\delta(x, \{p\}) \leq \varphi'(p)$. In particular, there exists a $\varphi \in \mathcal{A}(x)$ such that $\varphi(p)= \bot$. Similarly, if $A = \{x\}$, then by Remark \ref{Transition} (ii), $\forall \varphi'' \in L^{X}$, $\delta(p, \{x\}) \leq \varphi''(x)$ and particularly, there exists a $\varphi \in \mathcal{A}(p)$ such that $\varphi(x)= \bot$.
	 
	 $(iv) \Rightarrow (ii)$ Suppose the condition holds, i.e., $\forall x \in X$ with $x \not = p$, there exists $\varphi \in \mathcal{A}(p)$ such that $\varphi (x)=\bot $ or  there exists $\varphi \in \mathcal{A}(x)$ such that $\varphi (p)=\bot $. By Remark \ref{Transition} (iii), $\forall x \in X$; $d'(x, p) \in \mathcal{A}(x)$, by Definition \ref{L-App Base}, $\forall x \in X$, $\alpha \lhd \top$, $\bot \prec \omega$ there exists $\varphi \in \mathcal{B}(x)$ such that $\varphi(p)*\alpha \leq d'(x, p) \vee \omega$. Since $\varphi(p) = \bot$ and $\bot * \alpha = \bot$, $\bot \leq d'(x, p) \vee \omega$, and in particular, there exists a $d \in \mathcal{G}$ such that $d(x, p) = \bot$. In a similar way, there exists a $d \in \mathcal{G}$ such that $d(p, x) = \bot$.
	 	
\end{proof}

\begin{theorem}
	Let $(X,\mathcal{G})$ be an L-gauge space and $p \in X$. $(X,\mathcal{G})$ is $T_{1}$ at $p$ if and only if for all $x\in X$ with $x\neq p$, there exists $d\in \mathcal{G}$ such that $d(x,p)=\bot =d(p,x)$.
\end{theorem}

\begin{proof}
	Suppose that $(X,\mathcal{G})$ is $T_{1}$ at $p$, $x\in X$ and $x\neq p$. Let $u=x_{1}$, $v=x_{2} \in X \vee_{p} X$.
	Note that
	\begin{equation*}
	d_{dis}(\nabla_{p}(u),\nabla_{p}(v))=d_{dis}(\nabla_{p}(x_{1}),\nabla_{p}(x_{2}))= d_{dis}(x,x)=\top
	\end{equation*}
	\begin{equation*}
	d(\pi_{1}S_{p}(u),\pi_{1}S_{p}(v))=d(\pi_{1}S_{p}(x_{1}),\pi_{1}S_{p}(x_{2}))= d(x,p)
	\end{equation*}
	\begin{equation*}
	d(\pi_{2}S_{p}(u),\pi_{2}S_{p}(v))=d(\pi_{2}S_{p}(x_{1}),\pi_{2}S_{p}(x_{2}))= d(x,x)=\top
	\end{equation*}
	where $d_{dis}$ is the discrete L-metric on $X \vee_{p} X$ and $\pi_{i}: X^{2} \rightarrow X$ are the projection maps for $i=1,2$. Since $u\neq v$ and $(X,\mathcal{G})$ is $T_{1}$ at $p$, by Lemma \ref{initialStr}
	\begin{eqnarray*}
		\bot &=&\bigwedge\ \{d_{dis}(\nabla_{p}(u), \nabla_{p}(v)),d(\pi_{1}S_{p}(u), \pi_{1}S_{p}(v)), d(\pi_{2}S_{p}(u),\pi_{2}S_{p}(v))\} \\
		&=&\bigwedge\ \{\top,d(x,p)\} \\
		&=& d(x,p)
	\end{eqnarray*}
	and consequently, $d(x,p)=\bot $.
	\newline
	
	Let $u=x_{2}$, $v=x_{1} \in X \vee_{p} X$. Similarly,
	\begin{equation*}
	d_{dis}(\nabla_{p}(u),\nabla_{p}(v))=d_{dis}(\nabla_{p}(x_{2}),\nabla_{p}(x_{1}))= d_{dis}(x,x)=\top
	\end{equation*}
	\begin{equation*}
	d(\pi_{1}S_{p}(u),\pi_{1}S_{p}(v))=d(\pi_{1}S_{p}(x_{2}),\pi_{1}S_{p}(x_{1}))= d(p,x)
	\end{equation*}
	\begin{equation*}
	d(\pi_{2}S_{p}(u),\pi_{2}S_{p}(v))=d(\pi_{2}S_{p}(x_{2}),\pi_{2}S_{p}(x_{1}))= d(x,x)=\top
	\end{equation*}
	
	It follows that
	\begin{eqnarray*}
		\bot &=&\bigwedge\ \{d_{dis}(\nabla_{p}(u), \nabla_{p}(v)),d(\pi_{1}S_{p}(u), \pi_{1}S_{p}(v)), d(\pi_{2}S_{p}(u),\pi_{2}S_{p}(v))\} \\
		&=&\bigwedge\ \{\top,d(p,x)\} \\
		&=& d(p,x)
	\end{eqnarray*}
	and consequently, $d(p,x)=\bot $.
	\newline
	
	Conversely, let $\overline{\mathcal{H}}$ be initial L-gauge basis on $X \vee_{p} X$ induced by $S_{p}: X \vee_{p} X \rightarrow U(X^{2}, \mathcal{G}^{2})=X^{2}$ and $\nabla_{p}: X \vee_{p} X \rightarrow U(X, \mathcal{G}_{dis})=X$ where, by Proposition \ref{gdis}, $\mathcal{G}_{dis}=\textbf{L-MET(X)} $ discrete L-gauge on $X$ and $\mathcal{G}^{2}$ is the product L-gauge structure on $X^{2}$ induced by $\pi_{i}:X^{2} \rightarrow X$ the projection maps for $i=1,2$. 
	
	Suppose for all $x\in X$ with $x\neq p$, there exists $d\in \mathcal{G}$ such that $d(x,p)=\bot =d(p,x)$. Let $\overline{d} \in \overline{\mathcal{H}}$ and $u, v \in X \vee_{p} X$.
	\newline
	
	If $u=v$, then
	\begin{eqnarray*}
		\overline{d}(u,v) &=&\bigwedge\ \{d_{dis}(\nabla _{p}(u),\nabla _{p}(u)),d(\pi
		_{1}S_{p}(u),\pi _{1}S_{p}(u)), \\
		& & \hspace{8mm} d(\pi _{2}S_{p}(u),\pi _{2}S_{p}(u))\} \\
		&=&\top
	\end{eqnarray*}

	If $u\neq v$ and $\nabla _{p}(u)\neq \nabla _{p}(v)$, then $d_{dis}(\nabla _{p}(u),\nabla _{p}(v))=\bot $ since $d_{dis}$ is discrete. By Lemma \ref{initialStr}
	\begin{eqnarray*}
		\overline{d}(u,v) &=&\bigwedge\ \{d_{dis}(\nabla _{p}(u),\nabla _{p}(v)),d(\pi
		_{1}S_{p}(u),\pi _{1}S_{p}(v)), \\ 
		& & \hspace{8mm} d(\pi _{2}S_{p}(u),\pi _{2}S_{p}(v))\} \\
		&=&\bigwedge\ \{\bot ,d(\pi _{1}S_{p}(u),\pi _{1}S_{p}(v)),d(\pi
		_{2}S_{p}(u),\pi _{2}S_{p}(v))\} \\
		&=&\bot
	\end{eqnarray*}

	Suppose $u\neq v$ and $\nabla _{p}(u)=\nabla _{p}(v)$. If $\nabla _{p}(u)=x=\nabla _{p}(v)$ for some $x\in X$ with $x\neq p$, then $u=x_{1}$ and $v=x_{2}$ or $u=x_{2}$ and $v=x_{1}$ since $u\neq v$.

	If $u=x_{1}$ and $v=x_{2}$, then by Lemma \ref{initialStr}
	\begin{eqnarray*}
	\overline{d}(u,v) &=&\overline{d}(x_{1},x_{2}) \\
	&=&\bigwedge\ \{d_{dis}(\nabla _{p}(x_{1}),\nabla _{p}(x_{2})),d(\pi
	_{1}S_{p}(x_{1}),\pi _{1}S_{p}(x_{2})), \\
	& & \hspace{8mm} d(\pi _{2}S_{p}(x_{1}),\pi_{2}S_{p}(x_{2}))\} \\
	&=&\bigwedge\ \{\top,d(x,p)\} \\
	&=& d(x,p) \\
	&=& \bot
	\end{eqnarray*}
	since $x \not = p$ and $d(x,p)=\bot $.
	\newline

	Similarly, if $u=x_{2}$ and $v=x_{1}$, then
	\begin{eqnarray*}
	\overline{d}(u,v) &=&\overline{d}(x_{2},x_{1}) \\
	&=&\bigwedge\ \{d_{dis}(\nabla _{p}(x_{2}),\nabla _{p}(x_{1})),d(\pi
	_{1}S_{p}(x_{2}),\pi _{1}S_{p}(x_{1})), \\
	& & \hspace{8mm} d(\pi _{2}S_{p}(x_{2}),\pi_{2}S_{p}(x_{1}))\} \\
	&=&\bigwedge\ \{\top,d(p,x)\} \\
	&=& d(p,x) \\
	&=& \bot
	\end{eqnarray*}
	since $x \not = p$ and  $d(p,x)=\bot $.
	\newline
	
	Hence,  for all $ u,v\in X\vee _{p}X$, we get
	\begin{equation*}
	\begin{array}{r}
	\overline{d}(u,v)=
	\begin{cases}
	\top, & \text{$u=v$} \\
	\bot, & \text{$u\neq v$}%
	\end{cases}
	\end{array}
	\end{equation*}
	and it follows that $\overline{d}$ is the discrete L-metric on $X\vee _{p}X$, i.e., $\overline{\mathcal{H}}=\{\overline{d}\}$, which means $\mathcal{G}_{dis}=\textbf{L-MET(X)} $. By Definition \ref{T0veT1Def} (ii), $(X,\mathcal{G})$ is $T_{1}$ at $p$.
	
\end{proof}

\begin{theorem} \label{AllQuantaleT1App}
	Let $(X, \mathfrak{B})$ be a $L$-approach space and $p \in X$. Then, the followings are equivalent.
	\begin{enumerate}
		\item[(i)] $(X, \mathfrak{B})$ is $T_{1}$ at $p$.
		\item[(ii)] $\forall$ $x \in X$ with $x \not= p$, there exists $d\in \mathcal{G}$ such that $d(x,p)=\bot=d(p,x)$.
		\item[(iii)] $\forall$ $x \in X$ with $x \not= p$, $\delta(x, \{p\})= \bot=\delta(p, \{x\})$.
		\item[(iv)] $\forall$ $x \in X$ with $x \not= p$, there exists $\varphi \in \mathcal{A}(p)$ such that $\varphi (x)=\bot $ and there exists $\varphi \in \mathcal{A}(x)$ such that $\varphi (p)=\bot $.
	\end{enumerate}
\end{theorem}

\begin{proof}
It is analogous to the proof of Theorem \ref{AllQuantaleT0App}.
	\end{proof}

\begin{remark}
	If $L=[0, \infty]$ with order $\geq$ and operation as addition, local $\overline{T_{0}}$ $L$-approach spaces and local $T_{1}$ $L$-approach spaces reduce to classical local $\overline{T_{0}}$ and local $T_{1}$ approach spaces defined in \cite{Qasim, Qasim2}. 
\end{remark}

\begin{example} \label{T0 but not T1}
	Let $(L=[0, 1], \leq, *)$ be a triangular norm with binary operation $*$ defined as $\forall \alpha, \beta \in [0, 1], \alpha * \beta = \alpha \wedge \beta$ and named as minimum triangular norm. The triple $([0, 1], \leq, \wedge)$ is a quantale. Let $X=\{a, b,c\}$, $A \subset X$ and $\delta: X \times 2^{X} \rightarrow ([0, 1] ,\leq, \wedge)$ be a map defined by $\forall x \in X$, $\delta(x, \emptyset) = 0$, $\delta(x, A)=1$ if $x \in A$, $\delta(b, \{a\})=0=\delta(c, \{a\})$, $\delta(a,\{b\})=1/2=\delta(a, \{b,c\})$, $\delta(c, \{b\})=1/3=\delta(c, \{a,b\})$, $\delta(a, \{c\})=1/4$ and $\delta(b, \{c\})= 1/5 = \delta(b, \{a,c\})$. Clearly, $\delta(x, A)$ is a $L$-approach distance space. By Theorem \ref{AllQuantaleT0App}, $(X, \delta)$ is $\overline{T_{0}}$ at $p=a$ but neither $\overline{T_{0}}$ at $p=b$ nor $p=c$. Similarly, by Theorem \ref{AllQuantaleT1App}, $(X, \delta)$ is not $T_{1}$ at $p$, $\forall p \in X$.
\end{example}

\section{Closedness and D-connectedness}

Let $X$ be a set and $p$ be a point in $X$. The \textit{infinite wedge product} (``one point union of infinite sets") of $X$ at $p$, $\vee^{\infty}_{p} X$ is formed by taking countably disjoint families of $X$.

A point $x$ in $\vee^{\infty}_{p} X$ is denoted as $x_{i}$ if it lies in $i$th component.

Let $X^{\infty}= X \times X \times X \times...$	be the infinite cartesian product of $X$.

\begin{definition} (cf. \cite{BaranClosed})
	\begin{enumerate} 
		\item[(i)]  A map $A^{\infty}_{p}:\vee^{\infty}_{p} X \rightarrow X^{\infty} $ is called infinite principle axis map at $p$ if
		$A^{\infty}_{p}(x_{i})= (p,p,....,p,x,p,...)$.
		\item[(ii)] A map $\nabla^{\infty}_{p}:\vee^{\infty}_{p} X \rightarrow X^{\infty} $ is called infinite folding map $p$ if
		$\nabla^{\infty}_{p}(x_{i})= x$  for all $i \in I$.
	\end{enumerate}	
		
\end{definition}

\begin{definition} \label{closedVeconnect} (cf. \cite{BaranClosed})
	Let $U:\mathcal{E}\rightarrow \mathbf{Set}$ be
	topological, $X \in Ob(\mathcal{E})$ with $U(X)=Y$ and $p \in Y$.
	
	\begin{enumerate}
		\item [(i)] $\{p\}$ is closed iff the initial lift of the $U$%
		-source $\{A^{\infty}_{p}:\vee^{\infty}_{p} Y \rightarrow Y^{\infty}$ and $\nabla^{\infty}_{p}:\vee^{\infty}_{p} Y \rightarrow UD(Y^{\infty})= Y^{\infty}\}$ is discrete, where $D$ is the discrete functor.
		\item[(ii)] $X$ is $D$-connected iff any morphism from $X$ to any discrete object is constant.
	\end{enumerate}	
\end{definition}

\begin{theorem} \label{p is closed}
	\label{closed} Let $(X,\mathcal{G})$ be an L-gauge space and $ p\in X $.  $p$ is closed in $X$ if and only if for all $x\in X$ with $x\neq p$, there exists $ d\in \mathcal{G}$ such that $d(x,p)=\bot $ or $d(p,x)=\bot $.
\end{theorem}

\begin{proof}
	Let $(X,\mathcal{G})$ be an L-gauge space, $ p\in X $ and $p$ is closed in $X$. Let $\overline{\mathcal{G}}$ be initial L-gauge on $\vee_{p}^{\infty}X$ induced by $A_{p}^{\infty}:\vee_{p}^{\infty}X \rightarrow U(X^{\infty},\mathcal{G}^{*})=X^{\infty}$ and $\nabla _{p}^{\infty}: \vee_{p}^{\infty}X \rightarrow U(X,\mathcal{G}_{dis})=X$ where $\mathcal{G}_{dis}$ is discrete structure on $X$ and $\mathcal{G}^{*}$ be the product structure on $X^{\infty}$ induced by $\pi _{i}: X^{\infty}\rightarrow X$ $(i\in I)$ projection maps. Suppose that $\mathcal{H}_{dis}=\{d_{dis}\}$ is a basis for discrete L-gauge where $d_{dis}$ is the discrete L-metric on $X$. Let $\mathcal{H}$ be L-gauge basis of $\mathcal{G}$ and $d\in \mathcal{H}$, and $\overline{\mathcal{H}}=\{\overline{d}_{dis}\}$ be initial L-gauge basis of $\overline{\mathcal{G}}$ where $\overline{d}_{dis}$ is the discrete L-metric on $\vee_{p}^{\infty}X$. 
	
	We will show that for all $x\in X$ with $x\neq p$, there exists $ d\in \mathcal{G}$ such that $d(x,p)=\bot $ or $d(p,x)=\bot $. Suppose that $d(x,p)>\bot $ and $d(p,x)>\bot $ for all $d\in \mathcal{G}$ and $x\in X$ with $x\neq p$. For $i,j,k\in I$ with $i\neq j$ and $i\neq k\neq j$, note that 
	\begin{equation*}
	d_{dis}(\nabla _{p}^{\infty}(x_{i}),\nabla _{p}^{\infty}(x_{j}))=d_{dis}(x,x)=\top
	\end{equation*}
	\begin{equation*}
	d(\pi _{i}A_{p}^{\infty}(x_{i}),\pi _{i}A_{p}^{\infty}(x_{j}))=d(x,p)
	\end{equation*}
	\begin{equation*}
	d(\pi _{j}A_{p}^{\infty}(x_{i}),\pi _{j}A_{p}^{\infty}(x_{j}))=d(p,x)
	\end{equation*}
	\begin{equation*}
	d(\pi _{k}A_{p}^{\infty}(x_{i}),\pi _{k}A_{p}^{\infty}(x_{j}))=d(p,p)=\top
	\end{equation*}
	Since $x_{i}\neq x_{j}$ $(i\neq j)$ and $p$ is closed in $X$, by Lemma \ref{initialStr}
	\begin{eqnarray*}
		\overline{d}_{dis}(x_{i},x_{j}) &=& \bigwedge\ \{ d_{dis}(\nabla _{p}^{\infty}(x_{i}),\nabla _{p}^{\infty}(x_{j})), d(\pi _{i}A_{p}^{\infty}(x_{i}),\pi _{i}A_{p}^{\infty}(x_{j})), \\
		& & \hspace{8mm} d(\pi _{j}A_{p}^{\infty}(x_{i}),\pi _{j}A_{p}^{\infty}(x_{j})), d(\pi _{k}A_{p}^{\infty}(x_{i}),\pi _{k}A_{p}^{\infty}(x_{j}))  \} \\
		&=& \bigwedge\ \{\top,d(x,p),d(p,x)\} \\
		&>& \bot
	\end{eqnarray*}
	A contradiction to the fact that $\overline{d}_{dis}$ is the discrete L-metric on $\vee_{p}^{\infty}X$. Hence $d(x,p)=\bot $ or $d(p,x)=\bot $.
	\newline
	
	Conversely, let $\overline{\mathcal{H}}$ be initial L-gauge basis on $\vee_{p}^{\infty}X$ induced by $A_{p}^{\infty}:\vee_{p}^{\infty}X \rightarrow U(X^{\infty},\mathcal{G}^{*})=X^{\infty}$ and $\nabla _{p}^{\infty}: \vee_{p}^{\infty}X \rightarrow U(X,\mathcal{G}_{dis})=X$ where, by Proposition \ref{gdis}, $\mathcal{G}_{dis}=\textbf{L-MET(X)} $ discrete L-gauge on $X$ and $\mathcal{G}^{*}$ be the product structure on $X^{\infty}$ induced by $\pi _{i}: X^{\infty}\rightarrow X$ $(i\in I)$ projection maps.
	
	Suppose for all $x\in X$ with $x\neq p$, there exists $ d\in \mathcal{G}$ such that $d(x,p)=\bot $ or $d(p,x)=\bot $. Let $\overline{d} \in \overline{\mathcal{H}}$ and $u, v \in \vee_{p}^{\infty}X$.
	\newline

	If $u=v$, then for $i\in I$,
		\begin{eqnarray*}
		\overline{d}(u,v) &=& \bigwedge\ \{ d_{dis}(\nabla _{p}^{\infty}(u),\nabla _{p}^{\infty}(u)), d(\pi _{i}A_{p}^{\infty}(u),\pi _{i}A_{p}^{\infty}(u)) \} \\
		&=& \top
		\end{eqnarray*}

	If $u\neq v$ and $\nabla_{p}^{\infty}(u)\neq \nabla_{p}^{\infty}(v)$, then $d_{dis}(\nabla _{p}^{\infty}(u),\nabla _{p}^{\infty}(v))=\bot $ since $d_{dis}$ is discrete. By Lemma \ref{initialStr} for $i\in I$,
		\begin{eqnarray*}
		\overline{d}(u,v) &=& \bigwedge\ \{ d_{dis}(\nabla _{p}^{\infty}(u),\nabla _{p}^{\infty}(v)), d(\pi _{i}A_{p}^{\infty}(u),\pi _{i}A_{p}^{\infty}(v)) \} \\
		&=& \bigwedge\ \{ \bot, d(\pi _{i}A_{p}^{\infty}(u),\pi _{i}A_{p}^{\infty}(v)) \} \\	
		&=&\bot
		\end{eqnarray*}

	Suppose $u\neq v$ and $\nabla_{p}^{\infty}(u)=\nabla_{p}^{\infty}(v)$. If $\nabla_{p}^{\infty}(u)=x=\nabla _{p}^{\infty}(v)$ for some $x\in X$ with $x\neq p$, then $u=x_{i}$ and $v=x_{j}$ for $i,j\in I$ with $i\neq j$ since $u\neq v$. Let $u=x_{i}$, $v=x_{j}$ and $i,j,k\in I$ with $i\neq j$ and $i\neq k\neq j$.
		
		\begin{eqnarray*}
		d_{dis}(\nabla_{p}^{\infty}(u),\nabla_{p}^{\infty}(v)) &=& d_{dis}(\nabla_{p}^{\infty}(x_{i}),\nabla_{p}^{\infty}(x_{j})) \\
	 	&=&d_{dis}(x,x) \\
	 	&=&\top
		\end{eqnarray*}
		\begin{eqnarray*}
			d(\pi _{i}A_{p}^{\infty}(u),\pi _{i}A_{p}^{\infty}(v)) &=& d(\pi_{i}A_{p}^{\infty}(x_{i}),\pi_{i}A_{p}^{\infty}(x_{j})) \\
			&=&d(x,p)
		\end{eqnarray*}
		\begin{eqnarray*}
			d(\pi _{j}A_{p}^{\infty}(u),\pi _{j}A_{p}^{\infty}(v)) &=& d(\pi_{j}A_{p}^{\infty}(x_{i}),\pi_{j}A_{p}^{\infty}(x_{j})) \\
			&=&d(p,x)
		\end{eqnarray*}
	
		and
		\begin{eqnarray*}
			d(\pi _{k}A_{p}^{\infty}(u),\pi _{k}A_{p}^{\infty}(v)) &=& d(\pi_{k}A_{p}^{\infty}(x_{i}),\pi_{k}A_{p}^{\infty}(x_{j})) \\
			&=&d(p,p) \\
			&=&\top
		\end{eqnarray*}
	
	It follows that
		\begin{eqnarray*}
			\overline{d}(u,v) &=& \bigwedge\ \{ d_{dis}(\nabla _{p}^{\infty}(x_{i}),\nabla _{p}^{\infty}(x_{j})), d(\pi _{i}A_{p}^{\infty}(x_{i}),\pi _{i}A_{p}^{\infty}(x_{j})), \\
			& & \hspace{8mm} d(\pi _{j}A_{p}^{\infty}(x_{i}),\pi _{j}A_{p}^{\infty}(x_{j})), d(\pi _{k}A_{p}^{\infty}(x_{i}),\pi _{k}A_{p}^{\infty}(x_{j}))  \} \\
			&=& \bigwedge\ \{\top,d(x,p),d(p,x)\}
		\end{eqnarray*}
	By the assumption $d(x,p)$ $=\bot $ or $d(p,x)=\bot $, we get $\overline{d}(u,v)=\bot $.
	\newline
		
	Therefore, $\forall\ u,v\in \vee_{p}^{\infty}X$, we have
	\begin{equation*}
	\begin{array}{r}
	\overline{d}(u,v)=
	\begin{cases}
	\top , & \text{$u=v$} \\
	\bot , & \text{$u\neq v$}
	\end{cases}
	\end{array}
	\end{equation*}
	and by the assumption $\overline{d}$ is the discrete L-metric on $\vee_{p}^{\infty}X$, i.e., $\overline{\mathcal{H}}=\{\overline{d}\}$,
	which means $\mathcal{G}_{dis}=\textbf{L-MET(X)} $. By Definition \ref{closedVeconnect} (i), $p$ is closed in $X$.
\end{proof}

\begin{theorem} \label{T0 and closed are equal}
	Let $(X, \mathfrak{B})$ be a $L$-approach space and $p \in X$. Then, the followings are equivalent.
	\begin{enumerate}
		\item[(i)] $(X, \mathfrak{B})$ is $\overline{T_{0}}$ at $p$.
		\item[(ii)] $\{p\}$ is closed in $(X, \mathfrak{B})$.
	\end{enumerate}	
\end{theorem}

\begin{proof}
	It follows from Theorems \ref{local T0 L-App} and \ref{p is closed}.
\end{proof}

\begin{theorem}
	\label{dconnect} An L-gauge space $(X,\mathcal{G})$ is D-connected if and only if for each distinct points $x$ and $y$ in $X$, there exists $ d\in \mathcal{G}$ such that $d(x,y)=\top=d(y,x)$.
\end{theorem}

\begin{proof}
	Let $(X,\mathcal{G})$ be an L-gauge space, $(Y,\mathcal{G}_{dis})$ be a discrete L-gauge space with $CardY > 1$ and $f \colon (X,\mathcal{G}) \rightarrow (Y,\mathcal{G}_{dis})$ be a contraction. Suppose $(X,\mathcal{G})$ is D-connected, i.e., every contraction $f$ from $(X,\mathcal{G})$ to any discrete L-gauge space is constant. Since $f$ is a contraction map and therefore, by definition, $\forall\ d_{dis} \in \mathcal{G}_{dis}$, $d_{dis} \circ (f \times f) \in \mathcal{G} $, i.e., there exists $ d\in \mathcal{G}$ such that $d=d_{dis}(f,f)$. It follows that for each distinct points $x$ and $y$ in $X$, $d(x,y)=d_{dis}(f(x),f(y))$. $f$ is constant since $(X,\mathcal{G})$ is D-connected. Then, for $x, y \in X$, $f(x)=f(y)=\alpha \in Y$. So 
	\begin{eqnarray*}
		d(x,y)=d_{dis}(f(x),f(y))=d_{dis}(\alpha,\alpha)=\top
	\end{eqnarray*}
		
	and
	\begin{eqnarray*}
		d(y,x)=d_{dis}(f(y),f(x))=d_{dis}(\alpha,\alpha)=\top
	\end{eqnarray*}
	Hence if $(X,\mathcal{G})$ is D-connected, then there exists $ d\in \mathcal{G}$ such that $d(x,y)=\top=d(y,x)$.
	\newline
	
	Conversely, suppose that the condition holds, i.e., for each distinct points $x$ and $y$ in $X$, there exists $ d\in \mathcal{G}$ such that $d(x,y)=\top=d(y,x)$. We show that $(X,\mathcal{G})$ is D-connected. Let $f \colon (X,\mathcal{G}) \rightarrow (Y,\mathcal{G}_{dis})$ be a contraction and $(Y,\mathcal{G}_{dis})$ be a discrete L-gauge space. If $CardY = 1$, then $(X,\mathcal{G})$ is D-connected since $f$ is constant. Suppose that $CardY > 1$ and $f$ is not constant. Then, there exist distinct points $x$ and $y$ in $X$ such that $f(x) \neq f(y)$. So
	\begin{eqnarray*}
		d(x,y)=d_{dis}(f(x),f(y))=\bot
	\end{eqnarray*}
	
	and
	\begin{eqnarray*}
		d(y,x)=d_{dis}(f(y),f(x))=\bot
	\end{eqnarray*}
	 a contradiction since $d(x,y)=\top=d(y,x)$ for $x, y \in X$ with $x \neq y$. Hence, $f$ is constant, i.e., $(X,\mathcal{G})$ is D-connected.
	
\end{proof}

\begin{theorem} \label{App D-connected spaces}
	Let $(X, \mathfrak{B})$ be a $L$-approach space. Then, the followings are equivalent.
	\begin{enumerate}
		\item[(i)] $(X, \mathfrak{B})$ is $D$-connected
		\item[(ii)] $\forall$ $x,y \in X$ with $x \not= y$, there exists $d\in \mathcal{G}$ such that $d(x,y)=\top=d(y,x)$.
		\item[(iii)] $\forall$ $x,y \in X$ with $x \not= y$, $\delta(x, \{y\})= \top=\delta(y, \{x\})$.
		\item[(iv)] $\forall$ $x,y \in X$ with $x \not= y$, there exists $\varphi \in \mathcal{A}(x)$ such that $\varphi (y)=\top $  and there exists $\varphi \in \mathcal{A}(x)$ such that $\varphi (y)=\top $.
	\end{enumerate}	
\end{theorem}

\begin{proof}
	$(i) \Leftrightarrow (ii)$ follows from Theorem \ref{dconnect}. By using Remark \ref{Transition} (i) to (v), proof is straightforward. 
\end{proof}
\begin{example} \label{Example of D-connected}
	Let $L=[0,1]\cup \{\bot = -1, \top =3\}$ with order $(\leq)$ and $*= \wedge$ is a quantale. Suppose $X$ be a non-empty set and define $\delta : X \times 2^{X} \rightarrow (L, \leq, \wedge)$ as $\forall x\in X$, $A \subseteq X$,
		\begin{equation*}
	\begin{array}{r}
	\delta(x,A)=
	\begin{cases}
	3, & \text{$A \not = \emptyset$} \\
	-1, & \text{$A = \emptyset$}%
	\end{cases}
	\end{array}
	\end{equation*}. By Theorem \ref{App D-connected spaces}, $(X, \delta)$ is $D$-connected.
\end{example}
\begin{remark}
	\begin{enumerate}
		\item[(i)] In $\mathbf{Top}$ as well as in $\mathbf{CHY}$ (category of Cauchy spaces and Cauchy continuous maps) local $T_{1}$ and notion of closedness are equivalent \cite{BaranSepBalk, KulaCauchy}, and local $T_{1}$ implies local $\overline{T_{0}}$. However, in $\textbf{L-App}$, by Theorem \ref{T0 and closed are equal},  local $\overline{T_{0}}$ and notion of closedness are equivalent and by Example \ref{T0 but not T1}, local $T_{1}$ implies local $\overline{T_{0}}$ but converse is not true in general.
		\item[(ii)] In $\mathbf{CP}$ (category of pairs and pair preserving maps) and in $\mathbf{Prox}$ (category of proximity spaces and $p$-maps), local $\overline{T_{0}}$, local $T_{1}$ and notion of closedness are equivalent \cite{BaranClosed, KulaProx}.
		\item[(iii)] By Examples \ref{T0 but not T1} and \ref{Example of D-connected}, there is no relation between notion of closedness and D-connected objects, local $T_{1}$ and $D$-connected objects. More precisely, in $\textbf{L-App}$ as well as in $\mathbf{Prox}$ \cite{KulaProx}, we can say that 
		\begin{enumerate}
			\item[(a)] if $X$ is $D$-connected, then, it is not $\overline{T_{0}}$ (resp. $T_{1}$) at $p$, $\forall p \in X$.
			\item [(b)] if $X$ is $\overline{T_{0}}$ (resp. $T_{1}$) at $p$, then $X$ is not $D$-connected for a point. 
			\item[(c)] if $X$ is not $D$-connected then, for any $p \in X$, it may or may not be $\overline{T_{0}}$ (resp. $T_{1}$) at $p$.
		\end{enumerate}
		
	\end{enumerate}
\end{remark}

\end{document}